\documentclass[12pt]{amsart}

\usepackage{amsmath,amssymb}
\usepackage{color}
\usepackage{mathrsfs}
\usepackage{graphicx}
\usepackage{lineno}

\usepackage{calc}
\makeatletter

\newlength{\temp@wc@width}

\newlength{\temp@wc@height}

\newcommand{\widecheck}[1]{%
\setlength{\temp@wc@width}{\widthof{$#1$}}%
\setlength{\temp@wc@height}{\heightof{$#1$}}%
#1\hspace{-\temp@wc@width}%
\raisebox{\temp@wc@height+2pt}[\heightof{$\widehat{#1}$}]%
{\rotatebox[origin=c]{180}{\vbox to 0pt{\hbox{$\widehat{\hphantom{#1}}$}}}}%
}

\makeatother

\numberwithin{equation}{section}
\theoremstyle{plain}
\newtheorem{theorem}{Theorem}[section]
\newtheorem{lemma}[theorem]{Lemma}
\newtheorem{observation}[theorem]{Observation}

\newtheorem{corollary}[theorem]{Corollary}
\newtheorem{proposition}[theorem]{Proposition}
\newtheorem*{conjecture}{Conjecture}
 {\theoremstyle{definition}}

\def\dim{\operatorname{dim}}

\def\T{ \mathbb T}
\def\R{ \mathbb R}
\def\H{H^\infty}

\def\D{{ \mathbb D}}
\def\C{{ \mathbb C}}

\def\N{{ \mathbb N}}
\def\e{\varepsilon}

\def\ssc{\scriptscriptstyle}

\def\dis{\displaystyle}
\def\union{\cup}

\def\inter{\cap}
\def\Inter{\bigcap }
\def\ov{\overline}

\def\ss{\subseteq}
\def\emp{\emptyset}

\def\buildrel#1_#2^#3{\mathrel{\mathop{\kern 0pt#1}\limits_{#2}^{#3}}}

\def\BP{Blaschke product}
\def\CNBP{Carleson-Newman Blaschke product}
\def\IBP{interpolating Blaschke product}

\def\iff{\Longleftrightarrow}

\begin{document}

\title [The algebra $C(M(\H))_{\rm\ssc sym}$ ]{The covering dimension  of a distinguished subset
of the spectrum $M(\H)$ of $\H$  and the algebra of real-symmetric and continuous  functions
on  $M(\H)$}


\author{Raymond Mortini}
\address{\small D\'{e}partement de Math\'{e}matiques\\
\small LMAM,  UMR 7122,
\small Universit\'{e} Paul Verlaine\\
\small Ile du Saulcy\\
\small F-57045 Metz, France}

\email{mortini@univ-metz.fr}


\begin{abstract} 
We show that the covering dimension, $\dim E$, of  the closure $E$ of the interval $]-1,1[$
in the spectrum of $\H$ equals one. Using Su\'arez's result  that $\dim M(\H)=2$,
we then compute the Bass and topological stable ranks of the  algebra
$C(M(\H))_{\rm\ssc sym}$ of real-symmetric continuous functions  on $M(\H)$.
\end{abstract}

 \maketitle


\section*{Introduction}

In recent years the real counterparts to the classical complex function algebras
$A(\D), A(K),\H(\D)$ have gained  a certain interest due to their  appearance
in control theory. These are, for example,  the algebras
$$\mbox{$A(K)_{\rm\ssc sym}=\{f\in C(K),   f $ holomorphic in  $K^\circ$ and 
$f(z)=\ov{f(\ov z)}$ for all $z\in K\}$},$$
where $K$ is a real-symmetric compact set in $\C$ (that is $K$ satisfies  $z\in K\iff\ov z\in K$),
$$\mbox{$A_\R(\D)=\{f\in A(\D):  f$ real valued on $[-1,1]\}$},$$
and
$$\mbox{$\H_\R=\H_\R(\D)=\{f\in \H(\D):  f$ real valued on $]-1,1[\}$}$$
(see \cite{moru1,mw, rs1,rs2,wi1,wi2}). If $\mathbf D$ is the closed unit disk, then of course 
$A(\mathbf D)_{\rm\ssc sym}= A_\R(\D)$.  
The main feature in the papers referenced above was to give a determination
of the Bass and  topological stable ranks. In addition,  extension problems
to invertible tuples of real-symmetric functions in several complex variables
were studied in \cite{moru2} for the real algebras 
$$C(K)_{\rm\ssc sym}=\{f\in C(K): f( z_1, \dots, z_n)=\ov{f(\ov z_1, \dots, \ov z_n)}\}$$
 of complex valued continuous functions
 on real-symmetric compact sets $K$ in $\C^n$.

In the present work we will determine the topological and Bass stable ranks of
the algebra $C(M(\H))_{\rm\ssc sym}$ of all complex-valued 
continuous functions on the spectrum
  $M(\H)$ of $\H$ that satisfy $f(z)=\ov {f(\ov z)}$ in $\D=\{z\in \C: |z|<1\}$.
  Note that in view of the corona theorem, $\D$ can be viewed of as a dense subset of $M(\H)$.
  We will call $C(M(\H))_{\rm\ssc sym}$ the real-symmetric algebra associated with 
  $C(M(\H))$.
  Let us point out that the trace of $C(M(\H))$ in $\D$ is  a proper subalgebra of the algebra
  $C_b(\D,\C)$ of all  bounded, continuous and complex valued functions on $\D$.

  From a topological view point, the space $M(\H)$  is a very bizarre space;
  it is a non-metrizable, connected, compact Hausdorff space of cardinal  at least  $2^{\mathfrak c}$
  that is neither   locally connected, nor path-connected \cite{ag}. In particular, $M(\H)$ is 
  not contractible. Its covering dimension, though, is small: it is two (\cite{su}).
  
  A quite difficult problem is a concrete characterization of those continuous functions on 
  $\D$ that admit a continuous extension to $M(\H)$. K. Hoffman showed in his
  fundamental work \cite{ho} that $C(M(\H))$ is the smallest  uniformly closed 
  subalgebra of $C_b(\D,\C)$ that contains the (complex)-valued bounded harmonic
  functions. C. Bishop \cite{bish} showed that $f\in C_b(\D,\C)$ has a continuous extension
  to $M(\H)$ if and only if $f$ is uniformly continuous with respect to the hyperbolic metric 
  in $\D$ and for every $\e>0$ there is a Carleson contour $\Gamma$ in $\D$ so that $f$
  is within $\e$ of a constant  on each connected component of $\D\setminus \Gamma$.
  
  Henceforth, we give a thorough discussion of   the algebra $C(M(\H))_{\rm\ssc sym}$.
  It can be looked upon as a non trivial standard model for the classical real function algebras
  $C(X,\tau)$ presented for example in the monograph \cite{kuli} by Kulkarni and Limaye.
  
  \section{The algebra  $C(M(\H))_{\rm\ssc sym}$}
  
   We first look at  several properties of  the underlying space $M(\H) $ that are relevant
  to  the study of the algebra  $C(M(\H))_{\rm\ssc sym}$.
   Let $f\in C(M(\H))$ and define $f^*$ by $f^*(z)=\ov{f(\ov z)}$.  If $f\in \H$, 
   then $f^*\in \H$ and 
    the  operation $\sigma$ given by $\sigma( f)=f^*$ is an algebra involution on $\H$. 
    It is  well known that $\{f\in \H: f=f^*\}$ coincides with the algebra $\H_\R$ defined above.
    We shall now introduce the associated  involution on $C(M(\H))$.

  \begin{lemma}\label{mm*}
  For $m\in M(\H)$, let $m^*$ be defined as $m^*(f)=\ov{m(f^*)}$, $f\in \H$.
   Then $m^*\in M(\H)$.
   Moreover, if  $(\varphi_{z_\alpha})$ is a net of point functionals in $\D$ that 
   converges to $m\in M(\H)$,
then $(\varphi_{\ov z_\alpha})$ is a net of point functionals in $\D$
 that converges to $m^*\in M(\H)$.
  \end{lemma}
  
  \begin{proof}
It is obvious that $m^*$ is additive.  Moreover,
  $m^*$ is homogeneous because
  $$m^*(\lambda f)= \ov{ m((\lambda f)^*)}
=\ov{m(\ov \lambda f^*)}=\lambda \ov{m(f^*)}=\lambda m^*(f)$$
 whenever $f\in \H$. Thus $m^*\in M(\H)$.
 Now if $\varphi_{z_\alpha}\to m$, then 
$$\varphi_{\ov z_\alpha}(f)=f(\ov z_\alpha)=\ov{f^*(z_\alpha)}=\ov{\varphi_{z_\alpha}(f^*)}\to \ov{m(f^*)}.$$
\end{proof}

    \begin{lemma}\label{invol}
    Let $\tau_0:\D\to\D$ be the involution $a\mapsto \ov a$. Then $\tau_0$ admits
    a unique extension to a topological   involution $\tau$ between $M(\H)$ and itself.
    \end{lemma}
    \begin{proof}
    Let $\varphi_a:f\mapsto f(a)$ be the evaluation functional associated with $a\in\D$.
   For $m\in M(\H)$, consider the functional $m^*$ 
   given above. Define $\tau$ at $m$ by $\tau(m)=m^*$.  Note that $\tau(\varphi_a)=
    \varphi_{\ov a}$.
     Hence $\tau:M(\H)\to M(\H)$
  is an involution   between $M(\H)$ and itself.  It remains to show 
    that $\tau$ is continuous on $M(\H)$. So let $m_\alpha$ be a net in $M(\H)$ converging to $m$.
    Then for $f\in \H$
    $$\tau(m_\alpha) (f)=m_\alpha^*(f)=\ov{m_\alpha(f^*)}\to \ov{m(f^*)}=m^*(f)=\tau(m) (f).$$
    Thus $\tau$ is a topological involution extending   $\tau_0$.
     \end{proof}

     For a topological involution $\tau$ on  a compact Hausdorff space $X$ let
     $$\mbox{$C(X, \tau):=\{f\in C(X,\C): f(\tau(m))=\ov{f(m)}$ for any $m\in X\}$}$$
    be the classical real function algebra as given for example in \cite[p. 27]{kuli}. 
   Using Lemma \ref{invol} above,   we can now represent $C(M(\H))_{\rm\ssc sym}$
   as an algebra of this type:
   
     \begin{corollary}\label{repro}
     Let $\tau$ be the  involution from Lemma \ref{invol}. Then
     $$C(M(\H))_{\rm\ssc sym}=C(M(\H), \tau).$$
     \end{corollary}
     \begin{proof}
     Recall that $C(M(\H))_{\rm\ssc sym}$ was defined to be the set of all functions $f$
     in $C(M(\H))$ such that $f(\ov z)=\ov{f(z)}$ for all $z\in \D$. Let $f\in C(M(\H))_{\rm\ssc sym}$
     and 
     $m\in M(\H)$.  The continuity of $\tau$ implies that $g\circ\tau\in C(M(\H))$ whenever 
     $g\in C(M(\H))$.
      If $m$ is not point evaluation at some point in $\D$  then,
     by the corona theorem, we choose  a net $z_\alpha$ in $\D$ such that
      $\varphi_{z_\alpha}\to m$. Then, by Lemma \ref{invol}
      $$f(\tau (m))= \lim f(\tau(\varphi_{z_\alpha}))=\lim f(\ov z_\alpha)=
      \lim \ov{f(z_\alpha)}=\ov{f(m)}.$$
      So $C(M(\H))_{\rm\ssc sym}\ss C(M(\H), \tau)$. The other inclusion is trivial
      noticing that $\tau$ restricted to $\D$ is $\tau_0$.
     \end{proof}
     
     \begin{observation}\label{obs}
     $f\in C_b(\D,\C)$ has a continuous extension $F$ to
  $M(\H)$ if and only if $f^*$ has. 
     \end{observation}
  \begin{proof} 
    This follows 
 from the representation $f^*=(\ov{F\circ \tau})|_{\D}$ and  the fact that
     $\tau$ is continuous on $M(\H)$ (Lemma \ref{invol}).
     Another way to see this is to use 
  Hoffman's theory that states that $C(M(\H))$ is the  uniformly closed subalgebra
     of $C_b(\D,\C)$ generated by bounded holomorphic functions and their complex conjugates.
      \end{proof}
       
       In conformity with our previous notation, we keep on writing $f^*$ for the function
       $ \ov{f\circ \tau}$, whenever $f\in C(M(\H))$; that is 
       $$f^*(m)=\ov{ f(m^*)},$$
       where $m\in M(\H)$.
       
   In view of Corollary \ref{repro} and \cite[Theorem 1.3.20]{kuli}
      we have the following result on the structure of the
  maximal ideals of $C(M(\H))_{\ssc\rm sym}$ and their associated  multiplicative linear functionals
  (see also  \cite{mo3}  for the case of the algebra $A_\R(\D)$).
  
   \begin{theorem}
  Let $F_\tau= \{m\in M(\H): \tau(m)=m\}$ be the set of fixed points of $\tau$.
  Then  the following assertions hold:
  
   i) An ideal $I$ in $C(M(\H))_{\rm\ssc sym}$ is maximal  if and only if 
   $$I=I_m:=\{f\in C(M(\H))_{\rm\ssc sym}: f(m)=0\}$$
    for some $m\in M(\H)$. Moreover, $I_m=I_{m^*}$ for any $m$.
   
   ii)  $I_m$ has co-dimension 1 (in the real vector space $C(M(\H))_{\rm\ssc sym}$)
   if and only if  $m\in F_\tau$;
   
   iii) $I_m$ has co-dimension 2 if and only if $m\in M(\H)\setminus F_\tau$. 
      
   iv) The only multiplicative $\R$-linear functionals $\phi: C(M(\H))_{\rm\ssc sym}\to\R$
   are given by $\phi(f)=f(m)$, where $m\in F_\tau$. Their kernels are those maximal ideals
   $I_m$ that have co-dimension 1 in  the real vector space $C(M(\H))_{\rm\ssc sym}$.
   
   v) The remaining  multiplicative $\R$-linear functionals  have target space $\C$, (regarded as an
   algebra over $\R$),  and are given by $\phi(f)=f(m)$ or $\phi(f)=\ov{f(m)}$, where
   $m\in M(\H)\setminus F_\tau$. 
   Their kernels are the maximal ideals $I_m$ that have co-dimension 2 in  the real vector space $C(M(\H))_{\rm\ssc sym}$.
  
  vi) $C(M(\H))$ is the complexification of $C(M(\H))_{\rm\ssc sym}$.
  Each $q\in C(M(\H))$ can be uniquely written  as $q=f+ig$, 
  where $f,g\in C(M(\H))_{\rm\ssc sym}.$
 Here $f=(q+q^*)/2$ and $g=(q-q^*)/(2i)$.
 
 vii) $\sigma(f)=f^*$ is a topological involution on $C(M(\H))$.
   \end{theorem}

\begin{proof}
For the proof, we just note that if $\tau(m)=m$, then the evaluation functional
$\phi_m$  on $C(M(\H))_{\rm\ssc sym}=C(M(\H),\tau)$ satisfies
$$\phi_m(f)=f(m)=f(\tau(m))=\ov{f(m)}.$$
Hence $\phi_m$ is real valued and so the kernel has co-dimension 1.

On the other hand, if $\tau(m)\not=m$, then there exists (by \cite[Lemma 1.3.7]{kuli})
a function $f\in  C(M(\H), \tau)$ with $f(m)=i$ and $f(\tau(m))=-i$.
Thus the evaluation functional $\phi_m$ is a surjection onto the real algebra $\C$;
hence its kernel has codimension 2. 

The results now follow from   \cite[Theorem 1.3.20]{kuli}).
\end{proof}

That the maximal ideal spaces of $C(M(\H))_{\rm\ssc sym}$ and $C(M(\H))$
can be identified also  follows from  the fact that if 
$(f_1,\dots, f_N)\in C(M(\H))_{\rm\ssc sym}^N$,
then    a solution to the B\'ezout equation $\sum_{j=1}^N  q_j f_j=1$  in $C(M(\H))$
yields the   solution $\sum_{j=1}^N \frac{q_j+q_j^*}{2} f_j=1$ of the associated
B\'ezout equation  in $C(M(\H))_{\rm\ssc sym}$.

In the same way,     
we may identify the  maximal space of  the  real subalgebra  $\H_\R$  of 
$C(M(\H))_{\rm\ssc sym}$ with $M(\H)$.  Note, however, that if $m$ is a character of $\H_\R$, 
 then the  maximal ideals ${\rm Ker}\, m$ and ${\rm Ker}\, m^*$
 coincide even   in the case where $m\not=m^*$.

   \medskip
In the next section we will determine the set $F_\tau$ of fixed points of $\tau$. 
    
     \section{ The closure of the open unit interval in $M(\H)$}
     
     Let $E$ be the closure of $]-1,1[$ in $M(\H)$, 
     $M^+$ the closure of $\D^+:=\{z\in \D: {\rm Im}\; z>0\}$ in $M(\H)$ and
     $M^-$ the closure of $\D^-:=\{z\in \D: {\rm Im}\; z<0\}$ in $M(\H)$. 
       Finally, $\T^+=\{e^{i\theta}: 0<\theta<\pi\}$ and  $\T^-=\{e^{i\theta}: -\pi<\theta<0\}$.
     
       The goal in this section is to prove that $M^+\inter M^-=E$ and to show that $E=F_\tau$.
       To this end, we need a couple of lemmas.
     
     For $f\in C(M(\H))$, we denote by $Z(f)=\{m\in M(\H): f(m)=0\}$ the zero set of $f$.
     The set of points in $M(\H)$ with non-trivial Gleason parts will be denoted
     as usual, by $G$ (see \cite{ga,ho}). The pseudohyperbolic distance on $\D$
     is given by $\rho(z,w)\dis =\left|\frac{z-w}{1-z \ov w}\right|$. 
     Its extension to $M(\H)$ is defined as
     $$\rho(x,m)=\sup\{|f(x)|: f\in \H, ||f||_\infty \leq 1, f(m)=0\},$$
     $x,m\in M(\H)$.

    \begin{observation}\label{obs2}
     Let $f\in C(M(\H))$. Then 
  $f^*(x)=\ov{f(x)}$ whenever $x\in E$. 
  \end{observation}
  \begin{proof}
  Let $(r_\alpha)$ be a net in $]-1,1[$
  that converges to $x$. Then
  $$f^*(x)=\lim f^*(r_\alpha)=\lim \ov{f(\ov r_\alpha)}=\lim \ov{f(r_\alpha)}=\ov{f(x)}.$$
\end{proof}

Our subsequent results will be based on the following  assertion.
Recall that for  a real or complex function algebra $A$ with character space $M(A)$
 a  compact set $C\ss M(A)$ is said to be $A$-convex if
$C$ coincides with its $A$-convex hull
$$\check C=\{m\in M(A): |\hat f(m)|\leq \max_C |\hat f|, ~~ \forall f\in A\},$$
where $\hat f$ denotes the Gelfand transform of $f$.

\begin{theorem}\label{convexity}
Let $S$ be  a closed subset of the closure $E$  of $]-1,1[$ in $M(\H)$.
Then $S$  is $\H$-convex as well as $\H_\R$-convex.
\end{theorem}
\begin{proof}

Let $x\notin S$. Since $\H$ is separating (see \cite[p. 242]{su}), there exists $f\in \H$
such that $f(x)=0$ and $f\not=0$ on $S$. We may assume that $||f||_\infty\leq 1$.
Let $g=ff^*$. Then $g\in \H_\R$ and so $g$ 
is real valued on $E$. By definition of $g$, we actually have that  $1\geq g\geq 0$ on $E$.
Since  $f\not=0$ on $S$,  we obtain from  $f^*=\ov f$ on $E$, that $\sigma:=\min_S g >0$.
Now let $h=1-g$. Then $h\in \H_\R\ss\H$, $h(x)=1$ and  so
$$\max_S |h| =\max_S h\leq 1-\sigma < |h(x)|.$$
Thus $x$ does not belong to the $\H_\R$-convex closure of $S$. Therefore $S$ is 
$\H_\R$-convex as well as $\H$-convex.
\end{proof}

\begin{lemma}\label{halfdisk}
Let $b$ be an \IBP\ all of whose zeros $z_n$ in $\D$ satisfy ${\rm Im}\, z_n<0$.
Suppose that $Z(b)\inter E=\emp$.  Then $Z(b)\inter M^+=\emp$ and $Z(b^*)\inter M^-=\emp$.
\end{lemma}

\begin{proof}
Let $x\in M(\H)$ satisfy $b(x)=0$. We may assume that $x\notin\D$. Then $x\in G$ and
$x$  belongs to the closure of the $\{z_n:n\in \N\}$ (see  \cite[p. 379]{ga}). 
Assuming that $x\in M^+={\rm cl}\, (\D^+) $, we get from Hoffman's result  \cite[p. 103]{ho} that 
$\rho(Z(b)\inter \D,  \D^+)=0$.   For $j\in \N$, let   $u_j\in \D^+$ and $n(j)$ be chosen so that $\rho(z_{n(j)}, u_j)\leq 1/j$. Then every cluster point $m$ of $\{z_{n(j)}:j\in\N\}$  belongs
to $Z(b)$. 

Note that ${\rm Im}\, z_{n(j)}<0$ and ${\rm Im}\, u_j> 0$. 
By passing to subnets, we may assume that $z_{n(j(\alpha))}\to m$.
Now, if ${\rm Im}\, a<0$ and ${\rm Im}\, \xi\geq 0$, then 
$$ \rho(a, {\rm Re}\; a)\leq \rho(a,\ov a)\leq \rho(a,\xi)+\rho(\xi,\ov a)\leq 2\rho(a,\xi).$$
Now letting $a=z_{n(j(\alpha))}$ and $\xi=u_{j(\alpha)}$, we obtain that
 $$\rho\bigl(z_{n(j(\alpha))}, {\rm Re}\,z_{n(j(\alpha))}\bigr)\to 0.$$
 By taking a further subnet, if necessary,  ${\rm Re}\,z_{n(j(\beta))}$ then converges
 to some $m'$. Note that this implies that $m'\in E$. Since $\rho$ is semi-continuous \cite[p. 103]{ho}, 
$\rho(m,m')=0$ and so $m=m'$. Thus $m\in E\inter Z(b)$.  Hence $Z(b)\inter E\not=\emp$;
a contradiction to our hypothesis.  Therefore $x\notin M^+$.  Since $x$ was an arbitrary zero of $b$,
we conclude that $Z(b)\inter M^+=\emp$.  Due to symmetry, we obviously have that
$Z(b^*)\inter M^-=\emp$, too. 
  \end{proof}

  \medskip
  
   Let us note that the previous result also holds for arbitrary \BP s  (see  Proposition \ref{blaschke}
   at the end of this section).     For the sake of completeness
   we present that result, too, although we will not use this fact in the present paper.
   The proof  itself is based on    Theorem \ref{inters} and on a
   factorization theorem  given by K.  Izuchi.

  Versions of the following  function theoretic lemma are  well known.   What we need here, 
are uniform estimates outside some cones.
For the reader's convenience we present its proof.

\begin{lemma}\label{harm}
Let $u$ be the harmonic function with boundary values $1$ on $\T^+$ and $0$ on $\T^-$
and let $C_\kappa$ be the cone
$$C_\kappa=\{z=x+iy\in \D: |y|\leq \kappa(1-x)\}.$$
Then  there exists $\sigma>0$ such that $1> u(z)\geq 3/4$ on  
$$h^+(C_\sigma):=\{z=x+iy\in \D, 0\leq \sigma (1-x)\leq y\}$$
and $0<u(z)\leq 1/4$ on
$$h^-(C_\sigma):=\{z=x+iy\in \D,  y<0, 0\leq \sigma (1-x)\leq |y|\}.$$
Moreover, $ u(r)=1/2$ for every $r\in \;]-1,1[$, $1/2\leq u\leq 1$ on $\D^+$ and
 $0\leq u\leq 1/2$ on $\D^-$.
\end{lemma}

\begin{proof}
Note that $u$ has the form 
$$u(re^{i\theta})=\frac{1}{2\pi}\int_0^\pi\frac{1-r^2}{1+r^2-2r\cos(t-\theta)}dt.$$
Now we use the following inequality:
$$ 1+r^2-2r\cos s=(1-r)^2+4r\sin^2(s/2)\leq (1-r)^2 +s^2.$$
Let $\theta \in [0, \pi/4]$. Then 
$$u(re^{i\theta})\geq \int_0^\theta +\int _\theta ^{\pi/2}\geq$$
$$\mbox{$\frac{1+r}{1-r}$}\frac{1}{2\pi}\int_0^\theta \frac{1}{1+\left( \frac{\theta-t}{1-r}\right)^2}dt +
\mbox{$\frac{1+r}{1-r}$}\frac{1}{2\pi}\int _\theta ^{\pi/2}\frac{1}{1+ \left( \frac{t-\theta}{1-r}\right)^2}dt=$$
$$-\frac{1+r}{2\pi} \arctan \left( \frac{\theta-t}{1-r} \right)\Bigl| ^\theta_0+
 \frac{1+r}{2\pi}\arctan \left( \frac{t-\theta}{1-r} \right)\Bigl|^{\pi/2}_\theta=$$
 $$\frac{1+r}{2\pi}\left[   \arctan\left( \frac{\theta}{1-r} \right)+ 
  \arctan\left( \frac{\frac{\pi}{2}-\theta}{1-r} \right) \right].$$
 
 Now if $re^{i\theta}$ stays outside the  cone 
 $$\mathscr C:=\{z=\tilde re^{i\tilde \theta}\in \D:  |\tilde \theta|< C(1-\tilde r)\},$$
 then $C(1-r)\leq \theta\leq \pi/4$ and hence
 $$u(re^{i\theta})\geq  
 \frac{1+r}{2\pi} \left[ \arctan C+  \arctan  \left(\frac{\pi/4}{1-r}\right)\right]
 \buildrel\longrightarrow_{}^{r\to 1}  
 \frac{\arctan C}{\pi} + \frac{1}{2}.$$
 Now $\mathscr C \ss C_\kappa$ with $\kappa=C=\tan\psi$, where $\psi$ is the angle
 between the horizontal axis and the line $y=\kappa (1-x)$ respectively the curve
 $r(\theta)=1-\frac{\theta}{C}$ at the point $1$. 
 
 Thus $u(re^{i\theta})\geq 3/4$ whenever $C$ is large,  $r=r(C)$ close to $1$,
 and $re^{i\theta}\in h^+(C_\sigma)$ for some $\sigma=\sigma(C)$.
 
A change of variable $-t\to s$  shows that $u(z)+u(\ov z)$ is the integral
over the Poisson kernel on the whole interval $[0,2\pi]$. Hence $u(z)+u(\ov z)=1$.
Now if $z\in h^+(C_\sigma)$, then  $\ov z\in h^-(C_\sigma)$ and so 
$$u(\ov z)= 1-u(z)\leq 1-3/4=1/4.$$

Finally, if $z=x$ is real, then $1=u(z)+u(\ov z)=2u(x)$; and so $u(x)=1/2$.

Next let $z\in \D^-$; that is $z=re^{i\theta}$ with $-\pi<\theta<0$.
Then $t-\theta \geq t$ and so
$$u(re^{i\theta})\leq \frac{1}{2\pi}\int_0^\pi\frac{1-r^2}{1+r^2-2r\cos(t)}dt=u(r)=1/2.$$

If $z\in \D^+$, then $u(z)=1-u(\ov z)\geq 1/2$.
\end{proof}

\begin{corollary}
Let $u$ be the harmonic function above. Then
\begin{enumerate}
\item $1/2\leq u\leq 1$ on $M^+$;
\item $0\leq u\leq 1/2$ on $M^-$
\item $u=1/2$ on $E$.
\end{enumerate}
\end{corollary}

{\parindent=0pt Question: do we have that $u=1/2$ exactly on $E$?\medskip}

Recall that for $\lambda\in \T$,  the fiber $M_\lambda$ is given by
$$M_\lambda=\{m\in M(\H): m(z)=\lambda\},$$
where $z$ denotes the identity function here.

\begin{theorem}\label{inters}
$M^+\inter M^-=E$.
\end{theorem}
\begin{proof}
First we note that $E\ss M^+\inter M^-$ by definition. Now
let $y\in M^+\inter M^-$. Of course, we may
 assume that $y\not\in\D$, since the fact that
$$M^+\inter M^-\inter\D=\;]-1,1[$$ is obvious.
We claim that  $y$ belongs to one of the two fibers $M_1$
or $M_{-1}$. In fact, suppose that $y\in M_\lambda$, where $\lambda\notin\{-1,1\}$.
We may assume that ${\rm Im}\; \lambda>0$.
Let $p_\lambda(z)=(1+\ov{\lambda} z)/2$ be a peak function in $A(\D)$ associated
with $\lambda$. Then $p_\lambda\equiv 1$ on $M_\lambda\ss M^+$, but
$|p_\lambda| \leq 1-\eta<1$ outside  small neighborhoods of $\lambda$ within $\mathbf D$.
In particular $|p_\lambda|\leq 1-\eta$ on $M^-$.
 Hence $y\notin  M^+\inter M^-$; a contradiction.

Let $x\in M(\H)\setminus E$. Since $M^+\union M^-=M(\H)$,  we may assume
that $x\in M^+$.   Also, by the paragraph above, we may assume that
$x\in M_1$. We claim that $x$ does not belong to $M^-$.

{\bf Case 1} $x\in G$. 

 Choose a closed neighborhood $U$ of $x$ in $M(\H)$
so that $ U\inter E=\emp$.
There exists an \IBP\ $b$ with $b(x)=0$ such that $Z(b)\inter \D\ss U^\circ\inter \D$.
Thus, by \cite[p.379]{ga}, $Z(b)\ss U$. 
Hence $Z(b)\inter E=\emp$.   Decompose $b$ in  a product $b=b_1b_2$
of two \IBP s, 
where the zeros of $b_1$ are those with imaginary part strictly positive and  where the zeros
of $b_2$ are those with imaginary part strictly negative.  Note that $b$ has no real zeros.

 Noticing that $Z(b_2)\inter E=\emp$, we obtain from
 Lemma \ref{halfdisk} that $Z(b_2)\inter M^+=\emp$. Thus $b_1(x)=0$.
 Again, since $Z(b_1)\inter E=\emp$ and the zeros of $b_1$ are contained in $M^+$,
 we obtain from Lemma \ref{halfdisk}, that $Z(b_1)\inter M^-=\emp$. Thus $x\notin M^-$.
 Hence, $M^+\inter M^-\inter G\ss E$. 
 \medskip
 
 {\bf Case 2} $x$ is a trivial point with $x\in M^+\inter M_1$.   
 
 Since the $M(\H)$-closure $S$ of every cone 
 $$C_\kappa=\{z=\xi+i\eta\in \D: \xi\geq 1/2, ~~ |\eta|\leq \kappa(1-\xi)\}$$
 is contained in $G$ (see \cite[p. 108]{ho}),  $x$ is not in $S$. By Lemma
 \ref{harm}, the harmonic function $u$ given there is bigger than $3/4$ on  $h^+(C_\kappa)$
 and smaller than $1/4$  on  $h^-(C_\kappa)$ 
if $C_\kappa $  has a sufficiently large opening. Now $ \ov{h^+(C_\kappa)}\ss M^+$,
$\ov{ h^-(C_\kappa)}\ss M^-$, $\ov{h^+(C_\kappa)}\inter \ov{h^-(C_\kappa)}=\emp$  and 
$$M_1=(\ov{h^+(C_\kappa)}\inter M_1) \union (S\inter M_1)
 \union (\ov{h^-(C_\kappa)}\inter M_1).$$
Since $x\in M^+$,  $u\geq 1/2$ on $M^+$ and $u\leq 1/4$ on $\ov{h^-(C_\kappa)}$,
 we deduce that $x\in h^+(C_\kappa)$ and so
 $u(x)\geq 3/4$.  Also, since $u\leq 1/2$ on $M^-$,
we conclude that $x\notin M^-$. 

 To sum up, we have shown that $(M^+\setminus E)\inter M^-=\emp$. 
 Therefore $M^+\inter M^-=E$.
 \end{proof}

 \begin{corollary}\label{fixpoints}
    Let $m\in M(\H)$. The following assertions hold:
    \begin{enumerate}
   
    \item
    $m\in M^+$ if and only if $m^*\in M^-$.
     \item
    The set $E$ coincides with the set of fixed points $F_\tau$ of $\tau$; 
    that is $m=m^*$ if and only if $m\in E$.
    \item
    $m\in E$ if and only if $\ov{m(f^*)}=m(f)$ for any $f\in \H$.
    \end{enumerate}
    \end{corollary}
   
     \begin{proof}
 (1)  Let  $m\in M^+$.
 By Lemma \ref{mm*}, if $z_\alpha\to m$, ${\rm Im}\;z_\alpha>0$,
  then $\ov z_\alpha\to m^*$. Thus $m^*\in M^-$.
 
  (2)  Let $m=m^*$. By (1), $m\in M^+\inter M^-$. Using Theorem \ref{inters},
  we conclude that $m\in E$.
  To prove the converse, let $m\in E$.  Choose  a net $r_\alpha\in \;]-1,1[$ converging to $m$.
  Then, by Lemma \ref{mm*}, $r_\alpha=\ov r_\alpha$ converges to $m^*$.  Thus $m=m^*$.
 Since $m^*=\tau(m)$,  it follows that $F_\tau=E$. 
 
  (3) This is merely a reformulation of the assertion that $m=m^*$.
    \end{proof}

 We add the following additional  information on $u$.
 
 \begin{proposition}
The following assertions hold:
\begin{enumerate}
 \item  $u\equiv 1$ on the set of trivial points in $M^+$;
 \item  $u\equiv 0$ on the set of trivial points in $M^-$.
 \end{enumerate}
 \end{proposition}
 \begin{proof}
If $x\in M^+\inter M_\lambda$, $\lambda\in\T\setminus\{-1,1\}$, then
$u(x)=1$ because $u$ is constant $1$ on $M_\lambda$. 
 If $x$ is   a trivial point in $M_1\inter M^+$, then $x$  lies outside the closure 
 of any cone. Hence  we can replace the 
 number $3/4$  in Lemma \ref{harm} by any number 
 $\sigma<1$  close to one. Thus $u(x)\geq \sigma$
 and so,  $u(x)=1$. A similar reasoning  holds for $x\in M^-$.
 \end{proof}
 
Using Theorem \ref{inters}. we may  generalize Lemma \ref{halfdisk} in the following way.
 
   \begin{proposition}\label{blaschke}
   Let $B$ be a \BP\ all of whose zeros $z_n$ in $\D$ satisfy ${\rm Im}\, z_n<0$.
Suppose that $Z(B)\inter E=\emp$.  Then $Z(B)\inter M^+=\emp$ and $Z(B^*)\inter M^-=\emp$.
   \end{proposition}
   
   \begin{proof}
 The hypothesis  $Z(B )\inter E=\emp$ and the fact that $M^+\inter M^-=E$ (Theorem \ref{inters})
     imply that
 $M^+\inter Z(B)$ and $M^-\inter Z(B)$ are disjoint, open-closed sets in $Z(B)$. 
   Suppose that both sets are nonempty.
  Then, by \cite[Theorem 2.1]{izu},    
  $B=B^+B^-$, where $Z(B^+)=M^+\inter Z(B)$, $Z(B^-)=M^-\inter Z(B)$.
  But $B$, and hence $B^+$, has no zeros in $\D^+$. Thus the factor $B^+$ does
  not exist. This contradiction shows that $Z(B)\inter M^+=\emp$.
    \end{proof} 
  
Finally we remark, that if $B$ is a \BP\ all of whose   zeros $z_n$ in $\D$ satisfy ${\rm Im}\, z_n<0$,
then  $Z( B)\inter M^+$ can be big, though. 
Just take the zeros $z_n$ in $\D^-$ with $\rho(z_n, 1-\frac{1}{n^2})\leq 1/n$ and let 
   $B$ be the associated \BP.
  Then $B(r)\to 0$ as $r\to 1$ and so  $B$ vanishes identically on every Gleason part $P(x)$
  associated with a point $x\in E\setminus \D$. We claim that, 
  $(\ov{P(x)}\inter M^+)\setminus E\not=\emp$
  and $(\ov{P(x)}\inter M^-)\setminus E\not=\emp$. In fact, suppose that 
  $r_\alpha\to x$, $r_\alpha\in ]0,1[$.
  Then $w_\alpha:=\frac {r_\alpha+ z}{1+r_\alpha z}\to L_x(z)$, 
 $w_\alpha\in \D^+$ for $z\in \D^+$, $\ov{w_\alpha}\to (L_x(z))^*$ and 
 $\ov{w_\alpha}\to L_x(\ov z)$. Since the Hoffman map $L_x$ is a bijection, 
 $L_x(\ov z)\not= L_x(z)$, and so $L_x(z)\not=(L_x(z))^*$. Thus, by Corollary \ref{fixpoints},
 $L_x(z)\in (P(x)\inter M^+)\setminus E$ and $L_x(\ov z)\in (P(x)\inter  M^-)\setminus E$.
  In particular, $\emp\not= P(x)\inter M^+\ss Z(B)$.
  
  Such a phenomenon does not occur when $b$ is an \IBP, since $Z(b)\ss M^-$ 
  whenever the zeros in $\D$ are in the lower half-disk. Thus the fact that $M^+\inter M^-=E$
  implies that no point in $M^+\setminus E$ can be  a zero of $b$.

\section{The covering  dimensions of $E$ and $M^+$}

First let us recall the definition of the notion of covering dimension (or \v Cech-Lebesgue
dimension) as given in \cite[p. 54]{eng} or \cite[p. 111]{pe}.
Let $X$ be  a normal topological space. Then $X$ is said to have dimension $n$, denoted by 
${\rm dim}\; X=n$,   if $n$ is the smallest integer such that every finite open covering
of $X$ has a finite open refinement of order $n$. Here, as usual, the order of  a
family $\mathcal A$  of subsets of $X$ is the largest integer $n$ such that 
$\mathcal A$ contains $n+1$ sets with a non-empty intersection.

In order to determine the covering dimension of $E$,  we need the following 
result from \cite[p. 119]{pe}. Recall  that a closed set $C$ {\sl separates} two
disjoint closed sets $E$ and $F$  in a normal space $X$ if $X\setminus C=G\union H$,
where $G$ and $H$ are two disjoint open sets with $E\ss G$ and $F\ss H$.

\begin{proposition}\label{dimension}
If $X$ is a normal space, the following assertions are equivalent:
\begin{enumerate}
\item  ${\rm dim}\, X\leq n$;
\item  For  each family of $n+1$ pairs of closed sets 
$$\{(E_1,F_1), \dots, (E_{n+1},F_{n+1})\}$$
where $E_i\inter F_i=\emp$, there  exists a family $\{C_1,\dots, C_{n+1}\}$
of closed sets such that $C_i$ separates $E_i$ and $F_i$ and $\Inter_{i=1}^{n+1} C_i=\emp$.
\end{enumerate} 
\end{proposition}

\begin{theorem}\label{cover}
a) Let $E$ be the closure of $]-1,1[$ in $M(\H)$. Then the covering dimension of $E$ is one.

b) The covering dimension of the closure, $M^+$, of $\{z\in \D: {\rm Im}\, z>0\}$ in $M(\H)$ is two.
\end{theorem}

\begin{proof}
a)  For $j=1,2$, let $(E_j,F_j)$ be two pairs of disjoint closed sets in $E$. By Theorem \ref{convexity},
the sets  $E_j\union F_j$ are $\H$-convex. So the maximal ideal space of  the
algebras $A_j=\ov{\H|_{E_j\union F_j}}$ equals $X:=E_j\union F_j$.  Since $E_j$
and $F_j$ are open-closed in $X$,  Shilov's idempotent theorem (see for example \cite[p. 88]{gam}),
yields a function $q_j\in A_j$ such that $q_j\equiv 1$ on $F_j$ and $q_j\equiv 0$ on $E_j$.
Thus there exists $f_j\in \H$  such that $f_j\sim 1$  on $F_j$ and $f_j\sim 0$ on $E_j$.
Let $h_j=f_jf_j^*$. Then $h_j\in \H_\R$ and $h_j$ is real valued on $]-1,1[$, hence on $E$. 
Moreover, since for $x\in E$ one has  $f^*(x)=\ov{f(x)}$ (\ref{obs2}), 
 $h_j$ is close to $1$ on $F_j$ and close to $0$ on $E_j$. Let
$k_j=2h_j-1$.Then $k_j\in \H_\R$ is real valued on $E$, too, and $k_j$ is close to $1$
on $F_j$ and close to $-1$ on $E_j$. 
Consider the pair $(k_1,k_2)$. Since $\H_\R$ has the topological stable rank 2 (\cite{mw}),
\footnote{for a definition see the next section}
there  is an invertible  pair $(g_1,g_2)$ of  functions  in $\H_\R$ so that $g_j$ and $k_j$
stay very close to each other. In particular, the  $g_j$ are real valued on $E$ and 
$g_j$ remains close to $-1$ on  $E_j$ and close to $1$ on $F_j$. But $Z(g_1)\inter Z(g_2)=\emp$.
Thus we may choose $C_j =Z(g_j)\inter E$ to conclude that $C_j$ separates $E_j$ and $F_j$,
(just take  $G_j=\{x\in E: g_j<0\}$ and $H_j=\{x\in E: g_j>0\}$.)
Hence, by Proposition \ref{dimension}, the covering dimension of $E$ is less than or equal to one.
The dimension cannot be zero, though, since $E$ is a continuum. Thus ${\rm dim}\, E=1$.\medskip

b) The fact that the covering dimension of the closure, $M^+$, of $\{z\in \D: {\rm Im}\, z>0\}$ in $M(\H)$ is two follows from Su\'arez's  result \cite{su} that ${\rm dim}\, M(\H)=2$
and the sum-property for the dimension \cite[p.42, Theorem 1.5.3 ]{eng}
that tells us  that if $X$ is the union of a finite (or countably infinite)  number of closed sets $X_j$
with ${\rm dim}\, X_j\leq d$, then ${\rm dim}\, X \leq d$.  Here we have $X=M^+\union M^-$
and, due to symmetry, ${\rm dim}\, M^+= {\rm dim}\, M^-$.
\end{proof}

Instead of using in the above proof the full power of the fact that ${\rm tsr}\,\H_\R=2$,
we can also prove part a) of Theorem \ref{cover}
 by applying the following Lemma.
\begin{lemma}

Let $(k_1,k_2)$ be a pair of  functions in $\H_\R$. Then, for every $\e>0$,
 there exists a pair $(b_1K_1, b_2K_2)$ of functions   in $\H_\R$
  such that 
 \begin{enumerate}
 \item the $b_j$ are \IBP s having only real zeros;
 \item $b_1$ and $b_2$ have no common zeros on $M(\H)$;
 \item $K_1$ and $K_2$ are zero free on $E$;
 \item $||b_jK_j-k_j||_E<\e$.
 \end{enumerate}
\end{lemma}
\begin{proof}
Let $k_j=B_jF_j$ be the Riesz factorization of $k_j$. Here $B_j$ is a \BP\ and $F_j$
is zero free on $\D$. Since $k_j\in \H_\R$, the zeros of $B_j$ are symmetric to the real axis and so
$B_j$, as well as $F_j$, belong to $\H_\R$. We may assume that $F_j\geq 0$ on $]-1,1[$. 
Then  for $\e>0$, the functions $F_j+\e$ have no zeros on $E$. Let $B_j=v_ju_j$, where
$v_j$ is the \BP\ formed with the real zeros of $B_j$. 
By \cite{mn}, the Frostman shifts $w_j:=\frac{v_j-\e}{1-\e v_j}$ are \CNBP s. 
Write $w_j$ as $w_j=d_je_j$, where $d_j$ is the factor  of $w_j$ formed with the real zeros.
Note that $w_j,d_j,e_j\in \H_\R$. Since $d_j$ is a \CNBP\ with real zeros only, it can be
uniformly approximated by \IBP s with real zeros. Let $W_j=e_ju_j$.
Due to the symmetry of the zeros,
 $W_j(a)=0$ if and only if $W_j(\ov a)=0$. Therefore, for $r\in \;]-1,1[$, 
 $$W_j(r)=\prod_{a: {\rm Im}\, a>0}\frac{\ov a}{|a|} \frac{a-r}{1-\ov a r}\cdot \frac{ a}{|\ov a|} \frac{\ov a-r}{1-a r}=\prod_{a: {\rm Im}\, a>0}\frac{|a-r|^2}{|1- a r|^2}.$$
Thus $W_j\geq 0$ on $]-1,1[$.
Hence $W_j+\e$ is zero free on $E$.
 Thus we are able to  approximate each $k_j$ by functions of the form ${\tilde b}_j K_j$, 
 where ${\tilde b}_j$ is an \IBP\ with real zeros only and where
 $$K_j=(F_j+\e)(W_j+\e).$$
 Let $b_1=\tilde b_1$.
 By moving those zeros of $\tilde b_2$ that are hyperbolically close to those of $\tilde b_1$,
 we may approximate $\tilde b_2$ by an \IBP\ $b_2$ so that $\inf_\D( |b_1|+|b_2|)\geq \delta>0$;
for example by replacing  $\tilde b_2=b_2^{(1)}b_2^{(2)}$  by 
the \IBP\  $b_2^{(1)}\frac{b_2^{(2)}-\e}{1-\e b_2^{(2)} }$. 
 The tuple $(b_1K_1,b_2K_2)$ is now the desired item.
\end{proof}


\section{The Bass and topological stable ranks for $C(M(\H))_{\rm\ssc sym}$}

In this section we determine some $K$-theoretic data for the algebra $C(M(\H))_{\rm\ssc sym}$.
Our construction  will use the following lemma.

  \begin{lemma}\label{reflection}
     Let $q\in C(M^+,\C)$. Suppose that $q$ is real-valued on $]-1,1[$.
     Then $q$ admits a unique  extension to $C(M(\H))_{\rm\ssc sym}$.
     \end{lemma}
     \begin{proof}
     Let $f$ be defined as
     $$f(m)=\begin{cases} q(m)& \text{if $m\in M^+,$}\\
     \ov{q(m^*)} & \text{ if $m\in M^-$.}
     \end{cases}
     $$
     Since $M^+\inter M^-=E$ (Theorem \ref{inters}) 
     and $m=m^*$ on $E$ (Corollary \ref{fixpoints}),
      the real valuedness of  $q$  on $]-1,1[$, hence
     on $E$, implies that $f$ is well defined.  Also, the continuity of $q$ on $M^+$
     implies the continuity of $m\mapsto  \ov{q(m^*)}$ whenever $m\in M^-$.
     In fact, let $m_\alpha$ be  a net  in $M^-$ converging to $m$.  Then 
     $m_\alpha^*=\tau(m_\alpha)$ converges to $\tau(m)=m^*$ by Lemma \ref{invol}. Hence,
using Corollary \ref{fixpoints}(1),
     $$\ov {q(m_\alpha^*)}\to \ov {q(m^*)}.$$
     Thus $f$ is continuous on $M(\H)$. Since  for $a\in\D$, $(\varphi_a)^*=\varphi_{\ov a}$,
     we obtain that $f\in C(M(\H))_{\rm\ssc sym}$.
     \end{proof}

  Let $A$ be  a commutative unital (real or complex) Banach
algebra with unit element denoted by 1.  The set of invertible $n$-tuples in $A$ is the set
$$U_n(A)=\{(f_1,\dots,f_n)\in A^n\; \bigl| \; \exists  g=(g_1,\dots, g_n)\in A^n:
 \sum_{j=1}^n f_jg_j=1\}.$$

An element  $(f_1,\dots, f_n, g)\in U_{n+1}(A)$ is said to be {\it reducible},  if
there exists $(x_1,\dots, x_n)\in A^n$ so that  
$$(f_1+x_1g,\dots, f_n+x_ng)\in U_n(A).$$

The smallest integer $n$ for which every element  in $U_{n+1}(A)$ is  reducible
is called the {\it Bass stable rank} of $A$ and is denoted by ${\rm bsr}(A)$.
If no such  integer exists, then  ${\rm bsr}(A)=\infty$.

A related concept is that of the {\it topological stable rank}, ${\rm tsr}(A)$, of $A$ (see \cite{r}). 
This is the smallest
integer $n$ such that  $U_n(A)$ is dense in $A^n$.
If no such $n$ exists, then ${\rm tsr}(A)=\infty$.  It is well known that 
${\rm bsr}(A)\leq {\rm tsr}(A)$ (see \cite{r,mowi}).

Many papers have dealt with the determination of the Bass and/or  topological stable rank for
concrete function algebras 
(see for instance  \cite{cl, cs1, cs2, jmw, mw, ru, ru2, ru1, rs1, rs2,  su1, tr}).
It has been shown by Vasershtein \cite{va} and Rieffel \cite{r} that whenever $X$ is a compact 
Hausdorff space, then

$$\mbox{${\rm tsr}(C(X,\C))={\rm bsr}(C(X,\C))= \left[\frac{{\rm dim}\; X}{2} \right]+1$}$$
and
$${\rm tsr}(C(X,\R))={\rm bsr}(C(X,\R))={\rm dim} \;X +1.$$

 The following result can now be deduced from Theorem \ref{cover}  and   Su\'arez's 
result  \cite{su} that  the covering dimension of $M(\H)$ is  2.

\begin{corollary}\label{restricted}\hfill

\begin{enumerate}
\item ${\rm  tsr}\, C(M(\H))={\rm  bsr}\, C(M(\H))=2$;
 \item ${\rm  tsr}\, C(E,\R)={\rm  bsr}\, C(E,\R)= 2$;
  \item  ${\rm  tsr}\, C(M^+,\C)={\rm  bsr}\, C(M^+,\C)=2$.
 \end {enumerate}
\end{corollary}

For an $n$-tuple $\mathbf f= (f_1,\dots,f_n)$ of complex-valued fucntions, let
$$|\mathbf f|=\Bigl(\sum_{j=1}^n |f_j|^2\Bigr)^{1/2}.$$ 
As usual,   $S^n$ denotes
the unit sphere 
$$\{(x_1,\dots, x_{n+1})\in \R^{n+1}: \sum_{j=1}^{n+1}x_j^2=1\}$$
in $\R^{n+1}$.   Finally, if $(z_1,z_2)\in \C^2$ with $|z_1|^2+|z_2|^2=1$,
then we say that $(z_1,z_2)\in S^3$.

We are now able to prove the main result of this paper.  For matter of comparison, recall that
${\rm  bsr}(\H)= 1$ (\cite{tr}),
${\rm  tsr}(\H)= 2$ (\cite{su1}) and
${\rm  bsr}(\H_\R)={\rm  tsr}(\H_\R)=2$ (\cite{mw}).

\begin{theorem}\label{stable}

  ${\rm  tsr}\, C(M(\H))_{\rm \ssc sym}={\rm  bsr}\, C(M(\H))_{\rm \ssc sym}=2$.

\end{theorem}
\begin{proof}

This follows as in \cite{moru1} by using that ${\rm  bsr}\, C(E,\R)=2$ and
 ${\rm  bsr}\, C(M^+,\C)=2$.
For the reader's convenience we present those parts that need replacing $\D$ by
$M(\H)$, and $\D^+$ by $M^+$. 

1.  We first note that ${\rm bsr}(C(M(\H))_{\rm\ssc sym})>1$, since
the invertible pair $(z,1-z^2)$ is not reducible.
\medskip
 
 2. Next we indicate how to prove that ${\rm tsr}(C(M(\H))_{\rm\ssc sym})\leq 2$.  
Let $\mathbf f=(f_1,f_2)\in (C(M(\H))_{\rm\ssc sym})^2$ and 
$$E_n=\{m\in M^+: |\mathbf f(m)|\geq 1/n\}.$$

{\bf Step 1} Suppose that $E_n\inter E\not=\emp$.
We claim that there is an $\R^2$-valued  extension of  the tuple
$\mathbf f/|\mathbf f| \in C(E_n\inter E, S^1)$
to $\mathbf {\tilde f_n}\in  C(E, S^1)$.
\medskip

To prove this, we choose  $g_n\in C(E,\R)$ with $g_n\equiv 0$ on 
$E_n\inter E$ and $g_n\equiv 1$ on $Z(f_1)\inter Z(f_2)\inter E$
(Urysohn's Lemma).

Then the triple $(f_1, f_2, g_n)$ is invertible
in $C(E,\R)$. Since by Corollary \ref{restricted} ${\rm bsr }(C(E,\R))=2$ , there  exist
 $h_{1,n}, h_{2,n}\in C(E,\R)$ such that 
 $$(f_{1}+ h_{1,n}g_n,  f_{2}+ h_{2,n}g_n)$$
 is invertible in $C(E,\R)$.  Now the pair
 $$\mathbf {\tilde f_n}:=(f_{1}+ h_{1,n}g_n,  f_{2}+ h_{2,n}g_n)/ 
| (f_{1}+ h_{1,n}g_n,  f_{2}+ h_{2,n}g_n)|$$
is the desired  extension. We point out that $\mathbf {\tilde f_n}$ is $\R^2$-valued.

{If $E_n\inter E=\emp$, then we let $\mathbf {\tilde f_n}=(1,0)$.}
\medskip

{\bf Step 2} Next we claim that there exists a $\C^2$-valued  extension of 
$\mathbf f/|\mathbf f| \in C(E_n, S^3)$
to $\mathbf {\hat f_n}\in  C( M^+, S^3)$ that coincides on $E$ with $\mathbf {\tilde f_n}$.
\medskip

In fact, define $\mathbf F_n= (F_{1,n}, F_{2,n})$ by

\begin{align}\mathbf F_n(m)=&\; \mathbf f(m)/|\mathbf f(m)| \text{~~whenever~~} m\in E_n,\\
\mathbf F_n(m)=&\; \mathbf{\tilde f_n}(m) \text{~~whenever~~} m\in E
\end{align}
and extended continuously to $M(\H)$ by Tietze. Note that $\mathbf F_n$ is well defined,
due to Step 1. Now let $G_n\in C(M^+,\R)$ be a real valued continuous function 
with $G_n\equiv 0$ on $E_n\union E$ and $G_n\equiv 1$ on $Z(F_{1,n})\inter Z(F_{2,n})$.
Then the triple $(F_{1,n}, F_{2,n}, G_n)$ is invertible in $C(M^+,\C)$.
Since  by Corollary \ref{restricted} ${\rm bsr}( C(M^+,\C))=2$, there exist 
$H_{1,n}, H_{2,n}\in C(M^+,\C)$ such that
$$(F_{1,n}+H_{1,n}G_n, F_{2,n}+H_{2,n}G_n)$$
is invertible in $C(M^+,\C)$. Now the pair
$$\mathbf {\hat f_n}= (F_{1,n}+H_{1,n}G_n, F_{2,n}+H_{2,n}G_n)/
 |(F_{1,n}+H_{1,n}G_n, F_{2,n}+H_{2,n}G_n)| $$
is the desired  extension.\medskip

{\bf Step 3} It is easy to check that 
  $|\mathbf f - (|\mathbf f| + 1/n)\, \mathbf {\hat f_n}|\leq 3/n$ on
$M^+$.\medskip

{\bf Step 4} In the steps above we have found a  $\C^2$-valued function 
$$\mathbf{g_n}:= (|\mathbf f| + 1/n)\, \mathbf {\hat f_n}$$
 with $|\mathbf f-\mathbf {g_n}|\leq 3/n$ on $M^+$.
Note that $\mathbf {g_n}$ is $\R^2$-valued on $E\supseteq ]-1,1[$. 
Thus by Lemma \ref{reflection}
we can use reflection to define
a $\C^2$-valued function $\mathbf\Phi_n$ on $M$ (whose components are in 
$C({M(\H)})_{\rm\ssc sym}$) so that $|\mathbf f-\mathbf\Phi_n|\leq 3/n$ on ${M(\H)}$
and such that $|\mathbf\Phi_n|\geq \frac{1}{n} >0$  on ${M(\H)}$.
  \end{proof}
  
  It remains an open problem  which pairs $(f,g)$ of functions in $C(M(\H))_{\rm\ssc sym}$
  are reducible. Recall that in $\H_\R$ an invertible pair $(f,g)$ is reducible
  if and only if $f$ has constant sign on the set $Z(g)\inter E$ (see \cite{wi2,wi3} and \cite{m}).
  The situation  in $C(M(\H))_{\rm\ssc sym}$ is more difficult, since a)  the behaviour 
  of $f$ outside $E$ is not determined by that in $E$ (in contrast to the analytic case)
  and b) the Bass stable rank of $C(M(\H))$ is two, and not one.
  So a characterization of the reducible elements in  $C(M(\H))_{\rm\ssc sym}$ 
  must also involve conditions  outside  $M(\H)\setminus E$.   A necessary condition
  for example is the following:
  
  Suppose that $(f,g)$ is reducible, say $u=f+hg\not=0$ on $M(\H)$ and
   let $C$ be a connected  component of  $M(\H)\setminus Z(g)$.
Suppose that  the closure of  $C$ is contained in $\D$. 
Then $u$ is a zero free (continuous) extension
of $f|_{\partial C}$ to $C$. Thus the Brouwer degree  of $f$ satisfies $d(f,C,0)=d(u,C,0)=0$.

Necessary and sufficient criteria for reducibility of individual pairs  in
$C(K)$ and $C(K)_{\rm\ssc sym}$, where $K\ss C$ is compact, 
have meanwhile  been developed (see  
  \cite{ru1} for preliminary material and \cite{moru3} for a full solution.

  \section{Conjecture}
  
  In view  of the results in this paper and the ones in \cite{moru2}, we conjecture that the following is true:

\begin{conjecture}
Let $X$ be a compact Hausdorff space, and $\tau$ a topological involution of $X$.
Denote the set of fixed points of $\tau$ by $E$.
Then 
$${\rm bsr}\; C(X,\tau)={\rm tsr}\; C(X,\tau)=\max\left\{\mbox{$\left[\frac{{\rm dim}\; X}{2}\right], {\rm dim}\; E$} \right\}+1.$$ 
\end{conjecture}
\bigskip

{\bf Note added in proof}

Meanwhile this conjecture has been confirmed (see \cite{moru4}). \bigskip

{\bf Acknowledgements}
I thank Rudolf Rupp for his contributions to  joint work preceding  this paper,
without those Theorem \ref{stable} would not have come to live.
I also thank Amol Sasane and Brett Wick for several discussions concerning contractability
or non-contractability of the spectrum of $\H$.


\begin{thebibliography}{99}
\bibitem{ag} S. Axler, P. Gorkin. Sequences in the maximal ideal space of $\H$, 
Proc. Amer. Math. Soc.  108 (1990), 731-740.
\bibitem {bish} Ch. Bishop. Some characterizations of $C(\mathcal M)$,
Proc. Amer. Math. Soc.  124 (1996), 2695-2701.
\bibitem{cl} G. Corach, A. Larotonda.  Stable range in Banach algebras, J. Pure and Appl. Algebra
32 (1984), 289-300.

\bibitem{cs1} G. Corach, F.D. Su\'arez. Stable rank in holomorphic function algebras,
Illinois J. Math. 29 (1985), 627--639.
\bibitem{cs2} G. Corach, F.D. Su\'arez.  On the stable range of uniform algebras and $\H$.
Proc. Amer. Math. Soc.  98 (1986), 607--610. 
\bibitem{eng} R. Engelking. {\em Dimension Theory}, North Holland Publ. Comp., Amsterdam,
1978.
\bibitem{gam} T.W. Gamelin. {\em Uniform algebras}, Chelsea Pub. Company, New York 1984. 
\bibitem{ga} {J.B. Garnett}. {\em Bounded Analytic Functions}, 
Academic Press, New York, 1981.

\bibitem{ho} K. Hoffman.  Bounded analytic functions and Gleason parts, Ann. of Math. (2) 86 (1967), 74--111.
\bibitem{izu} K. Izuchi. Common zero sets of equivalent singular inner functions II,
Studia Math. 180 (2007), 133-142.

\bibitem{jmw} P. W. Jones, D. Marshall, T. Wolff. Stable rank of the disc algebra, Proc. Amer. Math. Soc. {96} (1986),  603--604.
 \bibitem{kuli} S.H. Kulkarni, B.V. Limaye. {\em Real Function algebras}, 
Marcel Dekker, New York, 1992.






\bibitem{mo3} R. Mortini.  A distinguished real Banach algebra. Proc. Indian Acad. Sci.
 (Math. Sci.) 119 (2009), 629-634.
 
 \bibitem{m} R. Mortini.  Reducibility of   function pairs in $\H_\R$,
 to appear in Algebra i Analyz  resp.  St. Petersburg Math J.
 
 \bibitem{mn}  R. Mortini, A. Nicolau.  Frostman shifts of inner functions.
 J. d'Analyse Math. 92 (2004), 285-326.
 
\bibitem{moru1} R. Mortini, R. Rupp.
Approximation by invertible elements and the generalized $E$-stable rank
for $A({\mathbf D})_\R$ and $C({\mathbf D})_{\rm\ssc sym}$,
Math. Scand. 109 (2011), 114-132.

\bibitem{moru2} R. Mortini, R. Rupp.
Real-symmetric extensions of invertible tuples of
multivariable continuous functions, 
to appear  in Complex Analysis and Operator Theory.

\bibitem{moru3} R. Mortini, R. Rupp.
The Bass stable rank for the real Banach algebra $A(K)_{\scriptscriptstyle \rm sym}$,
J. Funct. Analysis 261 (2011), 2214-2237.

\bibitem{moru4} R. Mortini, R. Rupp.
 Stable rank for the real function algebra $C(X,\tau)$,
 to appear in Indiana Univ. Math. J.



 \bibitem{mw} R. Mortini, B. Wick. The Bass and topological stable ranks of $H^\infty_\R(\mathbf D)$ and $A_\R(\mathbf D)$,  J. Reine Angew. Math. 636 (2009), 175-191.
 
\bibitem{mowi}  R. Mortini, B. Wick. Spectral characteristics and stable 
ranks for the Sarason algebra $\H+C$, 
Michigan Math. J.  59 (2010), 395-409.

\bibitem{pe}A.R.  Pears, {\em Dimension theory of general spaces},  Cambridge Univ. Press
London, New York, Melbourne, 1975.

\bibitem{r} M. Rieffel. Dimension  and stable rank in the $K$-theory of $C^*$-algebras,
Proc. London Math. Soc. 46 (1983), 301--333.

\bibitem{ru} R. Rupp.  Stable ranks of subalgebras of the disc algebra,  
Proc. Amer. Math. Soc.  108 (1990), 137--142.
\bibitem{ru2} R. Rupp. Stable rank of finitely generated algebras, Archiv Math. 55 (1990), 
438--444.

\bibitem{ru1} R. Rupp. Stable rank and boundary principle, Topology Appl. 40 (1991), 
307--316.

\bibitem{rs1} R. Rupp, A. Sasane. On the stable rank and reducibility in algebras of real symmetric functions,  Math. Nachrichten. 283  (2010), 1194--1206. 

\bibitem{rs2} R. Rupp, A. Sasane. Reducibility in $A_\R(K)$, $C_\R(K)$  and $A(K)$,
Canad. J. Math. 62 (2010), 646--667.

\bibitem{su} D. Su\'arez. \v Cech cohomology and covering dimension for the $\H$ maximal ideal space, J. Funct. Anal. 123 (1994), 233--263.



\bibitem{su1} D. Su\'arez.  Trivial Gleason parts and the topological stable rank of $\H$,
Amer. J. Math. 118 (1996), 879--904.
\bibitem{tr} S. Treil. The stable rank of $\H$ equals 1, J. Funct. Anal. 109 (1992), 130--154.
 

 \bibitem{va}
L. Vasershtein.
Stable rank of rings and dimensionality of topological spaces,
Funct. Anal. Appl. {\bf 5} (1971), 102--110;
translation from Funkts. Anal. Prilozh.  5 (1971), No.2, 17--27.

\bibitem{wi1} B. Wick. A note about stabilization in $A_\R(\D)$, Math. Nachrichten 282 (2009),
912--916 
\bibitem{wi2} B. Wick.  Stabilization in $\H_\R(\D)$, Publ. Mat.  54 (2010), 25--52.
\bibitem{wi3} B. Wick.  Corrigenda: "Stabilization in $H^\infty_{\R}(\D)$'' 
  Publ. Mat.  55  (2011),   251--260.

\end{thebibliography}
\end{document}